\documentclass{jgcc} 
 \pdfoutput=1
\usepackage{lastpage}
\jgccdoi{18}{2}{3}{17698}
\jgccheading{}{\pageref{LastPage}}{}{}{Mar.~11,~2026}{Jul.~1,~2026}{}


\usepackage{hyperref}
\usepackage{orcidlink}
\theoremstyle{plain} 

\newcommand{\SG}{\mathrm{SG}}
\newcommand{\ID}{\mathrm{ID}}
\newcommand{\ra}{\rightarrow}

\newcommand{\ms}{\mapsto}

\newcommand{\subquot}{sub-quot }
\newcommand{\N}{\mathbb N}
\newcommand{\F}{\mathbb F}
\newcommand{\FC}{\mathcal F}
\newcommand{\SC}{\mathcal S}
\renewcommand{\L}{\mathcal L}
\newcommand{\Z}{\mathbb Z}

\begin{document}
\title[Group identification of $p$-groups]
  {The group identification problem \\
   for $p$-groups of small order}
	
\author[B. Eick]{Bettina Eick\, \orcidlink{0000-0003-2884-6545} }	
\address{Technische Universität Braunschweig, Germany}	
\email{b.eick@tu-braunschweig.de}  

\author[H. Schanze]{Henrik Schanze\, \orcidlink{0000-0001-7621-7300} }	
\address{Technische Universität Braunschweig, Germany}	
\email{henrik.schanze@tu-braunschweig.de}  
		
		
\begin{abstract}
We investigate which group-theoretic invariants are powerful in 
distinguishing among non-isomorphic $p$-groups. Based on this, we 
devise an effective algorithm to solve the group identification 
problem for the $10\,494\,213$ groups of order $2^9$. We exhibit 
$56$ pairs of groups of order $2^9$ which are difficult to distinguish 
by invariants.
\end{abstract}
		
\maketitle
	
\section*{Introduction}
\label{intro}
	
Let $n \in \N$ and let $\L_n = (L_1, \ldots, L_{t(n)})$ be a complete and
irredundant list of all groups of order $n$, containing one representative 
for each isomorphism type. Given a group $G$ of order $n$, the {\em group 
identification problem} asks for the unique index $i$ so that $L_i \cong G$.
Solving this problem yields a test for isomorphism of two groups of order 
$n$, if a list $\L_n$ is given. Such lists are available in 
the SmallGroups Library \cite{BEO02, SmallGrp} for all $n \leq 2000$, 
excluding those of order $1024$.
	
For all orders $n \leq 2000$, excluding those of order $512, 1024$ and 
$1536$, the SmallGroups Library provides an effective solution for the 
group identification problem. Given a group $G$, the algorithm iteratively 
evaluates invariants of $G$ and with each evaluation it narrows down the 
list of possible candidates for $L_i \in \L_n$ with $L_i \cong G$ until this 
list has length one. We refer to Besche \& Eick \cite{BE99} for a description.
	
Our central aim here is to extend the solution of the group identification 
problem to groups of order $512$. The $10 \; 494 \; 213$ groups of this 
order are available in the SmallGroups Library, see also O'Brien 
\cite{OBr90} and  Eick \& O'Brien \cite{EOB99} for background on their
determination. Hence our aim is to find invariants of groups that are 
fast to evaluate and effective in distinguishing among the groups of 
order $512$. 
	
As a first step we explore invariants related to proper subgroups and 
quotients. We introduce the notion of {\em sibling tuples} for 
collections of pairwise non-isomorphic groups that are not distinguished 
by our subgroups and quotient invariants. We then consider invariants 
related to character theory and power maps. These split the sibling tuples
into clusters; we introduce the notion of {\em strong sibling tuples} for 
each such cluster. We observe that there are $56$ strong sibling $2$-tuples 
among the $10\,494\,213$ groups of order $2^9$.
	
We combined the ideas behind (strong) sibling tuples with the ideas of 
group identification for the groups of order dividing $256$ and obtain 
a new group identification algorithm for the groups of order $512$. Its 
implementation is available in the IdPGroup package \cite{IdPGroup} of
GAP \cite{Gap}.
	
This paper is structured as follows. Section \ref{invs} introduces our
invariants and evaluates them for $p$-groups of small order.  Section 
\ref{2groups} describes the groups of order dividing $2^9$ that are not 
distinguishable by these invariants. Section \ref{idgroup} introduces 
the algorithm to solve the group identification problem for groups of 
order $512$.  Section \ref{quest} lists open problems.
	
\section*{Acknowledgment}
\noindent 
The authors thank Heiko Dietrich for comments. Part of this work was 
done while the authors visited Monash University (Melbourne, Australia)
supported by a DAAD grant.
		
\section{Invariants for finite groups}
\label{invs}
	
In this section we introduce our invariants for finite $p$-groups and discuss 
their application to $p$-groups of small order. Throughout, we denote with
$\SG(o,i)$ the $i$th group of order $o$ in the SmallGroups Library. If the
group identification for groups of order $o$ is available in the 
SmallGroups Library, then we say that {\em $\ID$ is available} for the
order $o$; we write $\ID(G)$ for the pair $(o,i)$ with $G \cong \SG(o,i)$.
	
\subsection{Subgroups and quotients}
\label{siblings}

For a group $G$, let $\SC(G) = \{ U \leq G \mid 1 \neq U \neq G\}$ denote the 
set of all proper subgroups of $G$ and let $\FC(G) = \{ G/N \mid 1 \neq N 
\lhd G \}$ be the set of all proper quotients of $G$. Two groups $G$ and $H$ 
are {\em \subquot equivalent} if there exist bijections $\varphi : \SC(G) 
\ra \SC(H)$ and $\rho : \FC(G) \ra \FC(H)$ so that
\bigskip
	
\begin{enumerate}
\item[\rm (a)]
$U \cong \varphi(U)$ for each $U \in \SC(G)$ and $Q \cong \rho(Q)$ for
each $Q \in \FC(G)$; 
\item[\rm (b)] 
$\varphi$ preserves conjugacy; that is, $U$ is conjugate to $V$ in $G$
if and only if $\varphi(U)$ is conjugate to $\varphi(V)$ in $H$.
\item[\rm (c)] 
$\varphi$ maps the Frattini subgroup of $G$ and the proper subgroups of the 
lower and upper central series in $\SC(G)$ to the corresponding subgroups 
in $\SC(H)$.
\end{enumerate}
\bigskip
	
\noindent 
Our next aim is a practical algorithm to determine the clusters of \subquot
equivalent groups in a given list of $p$-groups. For this purpose we determine
a {\em fingerprint} $F(G)$ for a $p$-group $G$ so that two groups $G$ and $H$
are \subquot equivalent if and only if $F(G) = F(H)$ holds. 

An algorithm to determine the fingerprint $F(G)$ is exhibited in the 
following. This takes as input a group 
$G$ with $|G| = p^n$ and it assumes that $\ID$ is available for all groups of 
order dividing $p^{n-1}$.
\newpage
	
\begin{algo}{\bf Fingerprint:}
\begin{enumerate}
\item[(1)]
Compute the conjugacy classes $U_1^G, \ldots, U_r^G$ of proper subgroups 
of $G$; see the survey \cite{Hul12} or Section 10.4 in \cite{HEO} for
background.
\item[(2)]
Determine $( (\ID(U_i), |U_i^G|) \mid 1 \leq i \leq r)$ and sort this list.
\item[(3)]
Extract the non-trivial normal subgroups $N_1, \ldots, N_s$ from the list
determined in (1).
\item[(4)]
Determine $( \ID(G/N_i) \mid 1 \leq i \leq s )$ and sort this list.
\item[(5)]
Determine $( \ID(H) \mid H \in S(G))$, where $S(G)$ consists of the Fratttini
subgroup and the proper subgroups appearing as terms of the lower and upper
central series of $G$.
\item[(6)]
Return the concatenation of the results in (2), (4), (5).
\end{enumerate}
\end{algo} 
\bigskip
	
\noindent 
A {\em sibling $l$-tuple} is a list of $l$ pairwise non-isomorphic groups
that are pairwise \subquot equivalent; that is, any two groups in the list
have the same fingerprint. We also write pair for $2$-tuple.
	
\begin{thm}
\label{sibl}
There exist sibling pairs of orders $2^7$, $3^5$ and $p^4$ for all primes
$p>3$ and these orders are minimal with this property; that is, there are 
no sibling pairs of orders dividing $2^6$, $3^4$ or $p^3$ for any prime 
$p>3$.
\end{thm}
	
\begin{proof}
For the primes $2$ and $3$ this can easily be determined by inspection using 
the SmallGroups Library of GAP \cite{Gap} and Algorithm Fingerprint. For 
example, sibling pairs are $(\SG(2^7, 1317), \SG(2^7, 1322))$ and 
$(\SG(3^5, 19), \SG(3^5, 20))$. There are no sibling pairs among the groups 
of orders dividing $2^6$ and $3^4$, respectively. 
	
Let $p > 3$ be an arbitrary prime. It is not difficult to see that there
is no sibling pair among the groups of order dividing $p^3$ using the 
classification of these groups. We claim that $(\SG(p^4, 9), \SG(p^4, 10))$
is a sibling pair. The two groups in this pair are not isomorphic by the
classification of groups of order $p^4$. We show that they are \subquot
equivalent. First, we exhibit power-commutator presentations for these 
groups. Let $v$ be a primitive root modulo $p$ and $x \in \{1,v\}$. Define
$G_x$ as the group with generators $g_1, \ldots, g_4$ and relations
\begin{eqnarray*}
g_2^p = g_4^x \;\; \mbox{ and } \;\; g_i^p = 1 \; (i \neq 2) \\
{[g_2, g_1]} = g_3 \;\; \mbox{ and } \;\; {[g_3, g_1]} = g_4, \\
{[g_i, g_j]} = 1 \mbox{ for all other } i>j. 
\end{eqnarray*}
Then $\SG(p^4, 9) = G_1$ and $\SG(p^4, 10) = G_v$ and these groups are
non-isomorphic. We show that $F(G_1) = F(G_v)$.
	
First, the center $Z(G_x)$ is $Z(G_x) = \langle g_4 \rangle$ and
the Frattini subgroup $\Phi(G_x)$ is $\Phi(G_x) = \langle g_3, g_4
\rangle$. Every proper normal subgroup $N$ of $G_x$ contains the center 
and hence $G_x/N$ is independent of $x$. The upper and lower central 
series of $G_x$ are $G_x > \Phi(G_x)  > Z(G_x) > \{1\}$. Hence the
Frattini subgroup and proper subgroups of upper and lower central series 
are independent of $x$. 
	
It remains to consider the conjugacy classes of proper subgroups of $G_x$.
We list these classes in Table \ref{cls}. Each row of the table describes
a conjugacy class or, in two cases, a family of conjugacy classes. Each 
class is described by its length and a representative. The representative
is a subgroup of $G_x$ described by a polycyclic generating sequence, see 
Section 8.3 of \cite{HEO} for details. Additionally, the isomorphism type 
of each representative is exhibited; we use $C_n$ to denote the 
cyclic group of order $n$.
		
\begin{table}[htb]
\begin{center}
\begin{tabular}{|c|c|l|l|}
\hline
& class length & class representative(s) & isomorphism type \\
\hline
(1) &  $1$   & $\langle g_4 \rangle$ & $C_p$ \\
(2) &  $p$   & $\langle g_3 \rangle$ & $C_p$ \\
(3) &  $p^2$ & $\langle g_1 \rangle$ & $C_p$ \\
\hline
(4) & $p$ & $\langle g_1 g_2^a, g_4 \rangle$ 
      ($a \in \{1, \ldots, p-1\}$) & $C_{p^2}$ \\
(5) & $p$ & $ \langle g_2, g_4 \rangle$ & $C_{p^2}$ \\
(6) & $1$ & $\langle g_3, g_4 \rangle$ & $C_p \times C_p$ \\
(7) & $p$ & $\langle g_1, g_4 \rangle$ & $C_p \times C_p$ \\
\hline
(8) & $1$ & $\langle g_2, g_3, g_4 \rangle$ & $C_{p^2} \times C_p$ \\
(9) & $1$ & $\langle g_1, g_3, g_4 \rangle$ & $\SG(p^3,3)$ \\
(10)& $1$ & $\langle g_1 g_2^a, g_3, g_4 \rangle$ 
	($a \in \{1, \ldots, p-1\}$) & $\SG(p^3,4)$ \\
\hline
\end{tabular}
\end{center}
\caption{Conjugacy classes of subgroups of $G_x$ with $x \in \{1, v\}$}
\label{cls}
\end{table}
\newpage
		
We briefly comment on Table \ref{cls}.  First, note that each 
element $g$ of $G_x$ can be represented (uniquely) as word of the form 
$g = g_1^{e_1} \cdots g_4^{e_4}$ with $e_i \in \{0, \ldots, p-1\}$. If
$e_1 = \ldots = e_{i-1} = 0$, then $e_i$ is the leading exponent of $g$.
Each cyclic subgroup of $G$ is generated by an element with leading 
exponent $1$. Further, if $e_2 \neq 0$, then $g$ has order $p^2$, otherwise
$g$ has order dividing $p$. The classes (1)-(3) of cyclic subgroups 
of order $p$ follow readily now. The relations of $G_x$ imply that each 
subgroup of order $p^2$ of $G_x$ contains $Z(G_x) = \langle g_4 \rangle$. 
It thus corresponds to a cyclic subgroup of order $p$ in $G_x / Z(G_x) \cong 
\SG(p^3, 3)$. The classes (4)-(7) follow readily from this. Each subgroup 
of order $p^3$ is a maximal subgroup; thus it contains $\Phi(G_x) = 
\langle g_3, g_4 \rangle$ and is normal. Hence it has the form 
$\langle g_1^{e_1} g_2^{e_2}, g_3, g_4 \rangle$ with either $e_1 = 1$ and $e_2$ 
arbitrary, or $e_1 = 0$ and $e_2=1$. This leads to the classes (8)-(10) 
and completes the table in Table \ref{cls}.
\end{proof}

\noindent
The sibling tuples of orders $2^7, 2^8, 2^9$ will be discussed in more detail
in Section \ref{2groups}. Computation shows that there is exactly
one sibling $2$-tuple among the groups of order $3^5$, while there are $34$ 
sibling $2$-tuples and two sibling $3$-tuples among the groups of order $3^6$. 
Computational evidence also suggests the following conjecture.
	
\begin{conj}
Let $p > 3$ be prime. 
\begin{enumerate}
\item[\rm (a)]
Among the $15$ groups of order $p^4$, there is exactly one sibling $2$-tuple.
(The proof of Theorem \ref{sibl} shows that there is at least one.)
\item[\rm (b)]
Among the $2p + 2 \gcd(p-1,3) + \gcd(p-1,4) + 61$ groups of order $p^5$,
there are: two tuples of length $(p+1)/2$, two tuples of length 
$(p-1)/2$, one tuple of length $\gcd(p-1,4)$, one tuple of length 
$2$ and, if $\gcd(p-1,3) = 3$, then there are two tuples of length $3$.
\end{enumerate}
\end{conj}
	
\noindent 
We observe that the groups of a sibling pair need not to have isomorphic
subgroup lattices. As an example, let $G_1 = \SG(128, 1317)$ and $G_2 = 
\SG(128, 1322)$. Using GAP, it is not difficult to verify that these two 
groups are a sibling pair. We observe that a map $\varphi$ between 
their proper subgroups does not respect joins and meets of the lattice. 
More precisely, both groups $G_i$ have three maximal 
subgroups $M_{ij} = \SG(64, 70)$ (with $1 \leq j \leq 3$) and five 
normal subgroups $U_{ik} \cong \SG(32, 21)$ (with $1 \leq k \leq 5$). 
Their partial subgroup lattice restricted to these groups is exhibited
in Figure \ref{lat}, where the subgroups $U_{ik}$ not contained in any 
$M_{ij}$ are omitted.
\begin{figure}[htb]
\begin{center}
\begin{tikzpicture}[scale=0.7]
\node (190) at (0,-1)  [shape=circle, draw]{$G_1$};
\node (182) at (-2,-3) [shape=circle, draw]{$M_{11}$};
\node (184) at (-0,-3) [shape=circle, draw]{$M_{12}$};
\node (177) at (2,-3)  [shape=circle, draw]{$M_{13}$};
\node (159) at (2,-6)  [shape=circle, draw]{$U_{13}$};
\node (170) at (0,-6)  [shape=circle, draw]{$U_{12}$};
\node (154) at (-2,-6) [shape=circle, draw]{$U_{11}$};
\draw[-] (177) edge[bend left=0] (190);
\draw[-] (182) edge[bend left=0] (190);
\draw[-] (184) edge[bend left=0] (190);
\draw[-] (154) edge[bend left=0] (182);
\draw[-] (159) edge[bend left=0] (177);
\draw[-] (170) edge[bend left=0] (184);
\end{tikzpicture}
\vspace{2em}
	
\begin{tikzpicture}[scale=0.7]
\node (190) at (0,-1)  [shape=circle, draw]{$G_2$};
\node (178) at (-2,-3) [shape=circle, draw]{$M_{21}$};
\node (184) at (0,-3)  [shape=circle, draw]{$M_{22}$};
\node (181) at (2,-3)  [shape=circle, draw]{$M_{23}$};
\node (147) at (2,-6)  [shape=circle, draw]{$U_{22}$};
\node (170) at (-1,-6) [shape=circle, draw]{$U_{21}$};
\draw[-] (178) edge[bend left=0] (190);
\draw[-] (181) edge[bend left=0] (190);
\draw[-] (184) edge[bend left=0] (190);
\draw[-] (147) edge[bend left=0] (181);
\draw[-] (170) edge[bend left=0] (178);
\draw[-] (170) edge[bend left=0] (184);
\end{tikzpicture}
\end{center}
\caption{Partial subgroup lattices for $G_1 = \SG(128, 1317)$ and 
$G_2 = \SG(128, 1322)$}
\label{lat}
\end{figure}

\subsection{Character tables and Brauer pairs}
\label{secstrgsibl}

Two finite groups $G$ and $H$ with conjugacy classes $C(G)$ and $C(H)$, 
respectively, have equivalent character tables, if there exists a 
bijection $\tau : C(G) \ra C(H)$ which induces a bijection on the 
complex characters $\sigma : Irr(H) \ra Irr(G) : \chi \ms \chi \circ \tau$. 
Pairs of groups with equivalent character tables are easy to find; an 
example is given by $D_8$ and $Q_8$. The groups $G$ and $H$ are {\em 
Brauer equivalent} if their character tables including power maps are
equivalent; that is, they have equivalent character tables and 
additionally $\tau \circ \pi_n^G = \pi_n^H \circ \tau$ holds for each 
$n \in \Z$, where $\pi_n^G : C(G) \ra C(G)$ is the $n$th power map 
induced by $g \ms g^n$. 

A {\em Brauer pair} is a pair of non-isomorphic Brauer equivalent groups.
Brauer \cite{Brauer} initiated the search 
for Brauer pairs. Dade \cite{Dade} found the first examples of such pairs;
these first examples have order $p^7$.

An algorithm to check if two groups $G$ and $H$ are Brauer equivalent is 
available in GAP \cite{Gap}. This determines the character tables of $G$ 
and $H$ and then searches for a permutation on the conjugacy classes of 
$G$ that induces the desired equivalence of complex characters and power 
maps.  We cite the following results on Brauer pairs of small prime-power
order.

\begin{thm} 
{\bf (Skrzipczyk \cite{Skrzipczyk}, Eick \& Müller \cite{EMu06})} \\
There exist Brauer pairs of orders $2^8$, $3^6$ and $p^5$ for all primes
$p > 3$ and these orders are minimal with this property; that is, there 
are no Brauer pairs of orders dividing $2^7$, $3^5$ or $p^4$ for any prime
$p>3$.
\end{thm}

\noindent
A {\em strong sibling $l$-tuple} is a list of $l$ pairwise non-isomorphic
groups that are pairwise \subquot equivalent and pairwise a Brauer pair;
that is, a strong sibling $l$-tuple is a sibling $l$-tuple so that each
pair in the list is a Brauer pair.

\begin{thm}
\label{strgsibl}
There exist strong sibling pairs of orders $2^8$, $3^6$ and $p^5$ for 
all primes $p>3$ and these orders are minimal with this property; that is,
there are no strong sibling pairs of orders dividing $2^7$, $3^5$ or $p^4$ 
for any prime $p>3$.
\end{thm}

\begin{proof}
Among the $155$ groups contained in sibling tuples of groups of order 
$2^8$, there are $8$ strong sibling pairs. Note that Skrzipczyk 
\cite{Skrzipczyk} reports that there are $10$ Brauer pairs of order 
$2^8$. Among the $74$ groups contained in sibling tuples of groups of 
order $3^6$, there are $16$ strong sibling pairs and $2$ strong sibling
$3$-tuples. These results can be obtained readily using GAP.
	
Let $p > 3$ be prime. Then there are no strong sibling pairs of orders dividing
$p^4$, since there are no Brauer pairs of such orders by \cite{EMu06}. Let 
$b = \gcd(p-1,4)$ and let $v$ be a primitive root mod $p$. For $x \in 
\{v, \ldots, v^b\}$ we define
\begin{eqnarray*}
G_x &=& \hspace{-0.2cm}
\langle g_1, \ldots, g_5 \; \mid \; g_1^p =g_5^x, \; \;
g_2^p=g_3^p=g_4^p=g_5^p=1, \;\; 
[g_i,g_j] = 1 \mbox{ for } i>j \mbox{ except } \\
&& \hspace{2cm} 
[g_2,g_1]=g_3, [g_3,g_1]=g_4, [g_4,g_1]=g_5, [g_3,g_2]=g_5 \rangle.
\end{eqnarray*}
Note that $b \in \{2,4\}$, since $p$ is an odd prime. We show that the 
groups defined as $G_x$ yield a strong sibling $b$-tuple. They are not
isomorphic by the classification of groups of order $p^5$, see Girnat 
\cite{Gir}, and they form a Brauer $b$-tuple by \cite{EMu06}. It remains 
to show that they are \subquot equivalent. 
First, $Z(G_x) = \langle g_5 \rangle$. Thus each proper normal subgroup of 
$G_x$ contains $g_5$ and the quotient is independent of $x$. The upper and 
lower central series of $G_x$ coincide and are
\[ G_x > \langle g_3, g_4, g_5 \rangle
> \langle g_4, g_5 \rangle
> \langle g_5 \rangle 
> \{1\}.\]
Hence all proper subgroups in this series are independent of $x$ and
similarly for $\Phi(G_x) = \langle g_3, g_4, g_5 \rangle$. It remains to
determine the conjugacy classes of proper subgroups of $G_x$. We use a
similar approach as in the proof of Theorem \ref{sibl}. The result is
exhibited in Table \ref{tab2}; compare with Table \ref{cls} for notation.
	
\begin{table}[htb]
\begin{center}
\begin{tabular}{|c|c|l|l|}
\hline
isom type & class length & representative(s) & isomorphism type \\
\hline
(1) &   $1$ & $\langle g_5 \rangle$ & $C_p$ \\
(2) &   $p$ & $\langle g_4 \rangle$ & $C_p$ \\
(3) & $p^2$ & $\langle g_3 \rangle$ & $C_p$ \\
(4) & $p^2$ & $\langle g_2 g_4^a \rangle$  
	($a \in \{0, \ldots, p-1\}$) & $C_p$ \\
\hline
(5) &   $1$ & $\langle g_4, g_5 \rangle$ & $C_p \times C_p$ \\
(6) &   $p$ & $\langle g_3, g_5 \rangle$ & $C_p \times C_p$ \\
(7) & $p^2$ & $\langle g_3, g_4 \rangle$ & $C_p \times C_p$ \\
(8) &   $p$ & $\langle g_2 g_4^a, g_5 \rangle$ 
	($a \in \{0, \ldots, p-1\}$) & $C_p \times C_p$ \\
(9) & $p^2$ & $\langle g_2 g_3^a, g_4 \rangle$ 
	($a \in \{0, \ldots, p-1\}$) & $C_p \times C_p$ \\
(10) & $p^2$ & $\langle g_1 g_2^a, g_5 \rangle$ 
	($a \in \{0, \ldots, p-1\}$) & $C_{p^2}$ \\
\hline
(11) & $1$ & $\langle g_3, g_4, g_5 \rangle$ 
		& $C_p \times C_p \times C_p$ \\
(12) & $p$ & $\langle g_2, g_4, g_5 \rangle$ 
		& $C_p \times C_p \times C_p$ \\
(13) & $p$ & $\langle g_2 g_4^a, g_3, g_5 \rangle$
	($a \in \{0, \ldots, p-1\}$) & $\SG(p^3, 3)$ \\
(14) & $p$ & $\langle g_1 g_2^a, g_4, g_5 \rangle$     
	($a \in \{0, \ldots, p-1\}$) & $\SG(p^3, 4)$ \\
\hline
(15) & $1$ & $\langle g_2, \ldots, g_5 \rangle$ 
		& $\SG(p^4,12)$ \\
(16) & $1$ & $\langle g_1 g_2^a, g_3, g_4, g_5 \rangle$ 
	($a \in \{0, \ldots, p-1\}$)
  	        & $\SG(p^4,8)$  \\
\hline
\end{tabular}
\end{center}
\caption{Conjugacy classes of proper subgroups of $G_x$ with $x \in 
	\{v, \ldots, v^b\}$}
\label{tab2}
\end{table}

We briefly comment on Table \ref{tab2}.
First, each cyclic subgroup of $G_x$ is generated by
an element $g_1^{e_1} \cdots g_5^{e_5}$ with leading exponent $1$. If
$e_1 \neq 0$, then the subgroup has order $p^2$ and contains $g_5$, and
otherwise the subgroup has order dividing $p$. The classes (1)-(4) follow
readily from this. Next, each of the classes containing $\langle g_5 \rangle$
correspond one-to-one to a class in $G/\langle g_5 \rangle$ and is 
independent of $x$. This yields the classes (5),(6),(8),(10)-(16).
It remains to consider the classes not containing $\langle g_5 \rangle$.
These are subgroups of $\langle g_2, \ldots, g_5 \rangle$ and hence are
independent of $x$. Further, they cannot contain both of $g_2, g_3$. From
this the classes (7) and (9) follow. In summary, all of these conjugacy 
classes are independent of $x$, and thus the result of the proof follows.
\end{proof}

\section{(Strong) sibling tuples of order dividing $2^9$}
\label{2groups}

We report on the (strong) sibling tuples of order dividing $2^9$. The 
following result is based on computation.

\begin{thm}
\label{2tothe9}
There are no sibling tuples of order dividing $2^6$.
\begin{enumerate}
\item[\rm (a)]
Among the $2\,328$ groups of order $2^7$ there are $3$ sibling $2$-tuples
and no strong sibling $2$-tuples.
\item[\rm (b)]
Among the $56\,092$ groups of order $2^8$ there are $71$ sibling $2$-tuples, 
$3$ sibling $3$-tuples and $1$ sibling $4$-tuple. There are $8$ strong 
sibling $2$-tuples.
\item[\rm (c)]
Among the $10\,494\,213$ groups of order $2^9$ there are $428$ sibling 
$2$-tuples, $6$ sibling $3$-tuples and $3$ sibling $4$-tuples. There are 
$56$ strong sibling $2$-tuples.
\end{enumerate}
\end{thm}

\noindent 
We consider the $56$ strong sibling pairs of order $2^9$ in more detail. 
Tables \ref{tab4} and \ref{tab5} in  Appendix \ref{app} give an 
overview on them, including their $\ID$s, their rank and $p$-class.
We also report the results of evaluating additional invariants on these
pairs. We briefly outline the considered invariants.

\subsection*{Isoclinism}

Isoclinism is an equivalence relation of finite groups, weaker than 
isomorphism. It was introduced by Hall \cite{Hall} as a method to assist 
the classification of $p$-groups. Two groups $G$ and $H$ are isoclinic if 
there is a pair of isomorphisms $\alpha : G/Z(G) \ra H/Z(H)$ and 
$\beta : [G,G] \ra [H,H]$ that is compatible with the commutator
map $\gamma_G : G/Z(G) \times G/Z(G) \ra [G,G] : (g Z(G), hZ(G)) \ra [g,h]$;
more precisely, $\gamma_H( \alpha(x), \alpha(y) ) = \beta( \gamma_G(x,y))$
holds for all $x,y \in G/Z(G)$. We use the algorithms of the XMod package 
\cite{XMod} to check isoclinism of groups in GAP.

\subsection*{Group rings and the modular isomorphism problem}

The modular isomorphism problem (MIP) asks if there exist two non-isomorphic
$p$-groups $G$ and $H$ with isomorphic group rings $\F G \cong \F H$, where
$\F$ is a field of characteristic $p$. First examples of such groups were 
found recently by Garcia-Lucas, Margolis \& del Rio \cite{Margolis}; the 
smallest examples are one pair of groups of order $2^9$.
An algorithm to check if $\F G \cong \F H$ for two $p$-groups $G$ and $H$
and $\F$ the field with $p$ elements was developed by Eick \cite{Eick08}
and implemented in \cite{modisom}. This algorithm determined that there 
are no examples of order dividing $2^8$ or $3^6$. Not all of the groups of 
order $2^9$ have been checked for MIP-examples, see Moede \& Margolis 
\cite{MoedeMargolis}. We check the strong sibling $2$-tuples of order 
$2^9$ in this respect, but found no new MIP examples.

\subsection*{(Outer) automorphism groups}

Two groups $G$ and $H$ are {\em (outer) automorphism equivalent} if 
$Aut(G) \cong Aut(H)$ or $Out(G) \cong Out(H)$, respectively.
Automorphism equivalent groups are not difficult to find: An example 
are the groups $D_8$ and $C_4 \times C_2$. 
An algorithm to compute the automorphism group of a finite $p$-group
was developed by Eick, Leedham-Green \& O'Brien \cite{ELGO02} and
implemented in the AutPGroup Package \cite{AutPGrp}. Given a
$p$-group $G$, the method returns a list of generators for $Aut(G)$;
if $Aut(G)$ is solvable, the list 
is a polycyclic generating set, see Section 8.3 of \cite{HEO} for details. 
This allows us to read off polycyclic presentations for $Aut(G)$ and $Out(G)$ 
which, in turn, facilitates isomorphism tests for these groups.

\section{Group identification for the groups of order $2^9$}
\label{idgroup}

The SmallGroups Library provides an effective algorithm to solve the group
identification problem for all contained group orders 
except those divisible by $2^9$. In particular, it has an algorithm to
identify the groups of order $2^8$. We used that as a starting point for our 
own investigations and the development of an algorithm to identify the
groups of order $2^9$. As a first step we recall the main ideas behind the
group identification for groups of order $2^8$.

\subsection{Group identification for the groups of order $2^8$}
\label{id28}

The algorithm for groups of order $2^8$ uses a decision tree. Each vertex 
$v$ in the tree corresponds to a set $s(v)$ of groups of order $2^8$. If 
$v$ is the root of the tree, then $s(v)$ is the set of all groups, and if 
$v$ is a leaf in the tree, then $|s(v)| = 1$. If $v$ is a vertex in the
tree with the graph-theoretic children $v_1, \ldots, v_l$, then $s(v)$ is 
the disjoint union of the sets $s(v_1), \ldots, s(v_l)$.

Each non-leaf vertex $v$ in the tree has an associated function $f_v$.
This function can be applied to the groups in $s(v)$ and evaluates an
invariant. If $u$ is a child of the vertex $v$ in the tree, then $u$ has 
an associated outcome $r(u)$ and it holds that
\[ s(u) = \{ G \in s(v) \mid f_v(G) = r(u)\}.\]

If a group $G$ of order $2^8$ is given, then the tree is used to decide 
$\ID(G)$ by iteratively evaluating $f_v(G)$ until the leaf with $s(v) = 
\{G\}$ is found. The leaf determines $\ID(G)$.
It remains to describe the function $f_v$ in more detail. If $v$ has
distance at most $3$ to the root, then $f_v$ depends on this distance
only. The function $f_v$ computes:
\begin{enumerate}
\item[$\bullet$] Distance $0$:
The abelian invariants of the quotients in the derived series of $G$.
\item[$\bullet$] Distance $1$:
The sizes of the quotients of the $p$-central series of $G$.
\item[$\bullet$] Distance $2$:
The orders $o(g)$ of the elements $g \in G$.
\item[$\bullet$] Distance $3$:
For each conjugacy class $g^G$ of $G$, the pair $(|g^G|, o(g))$. 
\end{enumerate}

\noindent 
The respective results of the last two computations are collected
to obtain an invariant.

\noindent 
Based on these invariants, $8349$ groups of the $56\,092$ groups of order
$2^8$ are identified.

For the vertices of distance at least $4$ to the root, individual 
invariants for each vertex are used. As the tree has a rather large 
number of vertices, no overall description of {\em all} used invariants 
is feasible. Instead, we include two important examples.
\bigskip

\noindent
{\bf Conjugacy class clusters and power-maps.}
Let $C$ be a conjugacy class of $G$, let $l(C)=|C|$ and let $o(C)$ be
the order of a representative of $C$. The $(\lambda,\omega)$-cluster 
of conjugacy
classes of $G$ is the union of all conjugacy classes $C$ of $G$ with 
$\lambda= l(C)$ and $\omega = o(C)$. Note that a conjugacy class cluster is 
invariant under isomorphisms.

The conjugacy class clusters $C_1, \ldots, C_k$ are now sorted and 
refined. For the refinement choose $r \in \N$ and evaluate $C_i^r 
= \{ g^r \mid g \in C_i\}$ for $1 \leq i \leq k$. For each $i, j 
\in \{1, \ldots, k\}$ let $A(i,j) = \{ g \in C_i \mid g^r \in C_j\}$ 
and $B(i,j) = C_i^r \cap C_j$. We then split the cluster $C_i$ into 
its subsets $A(i,1), \ldots, A(i,k)$ and we split $C_j$ into its 
subsets $B(1,j), \ldots, B(k,j)$. This refinement is still an 
invariant under isomorphisms.

To simplify notation we denote the refined conjugacy class clusters
by $C_1, \ldots, C_k$. We define the map $\varphi : \{1, \ldots, k\} 
\ra \{ 1, \ldots, k \}$ with $\varphi(i) = j$ if $C_i^r = C_j$. This 
map $\varphi$ is an isomorphism invariant.

This approach can be used with $r$th powers for any $r \in \N$. A 
variation of the approach is obtained if commutators are evaluated
instead of powers, or other words in the free group are considered
and evaluated. Hence the general idea described here allows a wealth 
of different functions $f_v$.

\begin{rem}
The result of the conjugacy class clustering allows us to identify all but
$38$ pairs of groups among the groups of order $2^8$.
\end{rem}
\bigskip

\noindent
{\bf Random isomorphism testing.}
A random isomorphism test for solvable groups of small order has been 
described in \cite{BE99}. This is applied in some cases when $s(v)$ has
two groups remaining and these are very hard to distinguish by invariants.
Let $G$ be the given group and let $s(v) = \{ G_1, G_2 \}$. 
The random isomorphism searches for an isomorphism $G \cong G_i$ for 
either $i=1$ or $i=2$ at random. Once an isomorphism is found, the group
$G$ is identified.
A more detailed description of the random isomorphism test is given in 
\cite{BE99}. It works well in practise, despite the randomised approach.

\begin{rem}
The result of this test is used for all but one of the remaining $38$ 
pairs. The last pair splits by checking if a specific group is
a subgroup of $G$.
\end{rem}

\subsection{Group identification for the groups of order $2^9$}
\label{id29}

Our algorithm for the identification of groups of order $2^9$ is organised 
as a decision tree similar to that for the groups of order $2^8$. We 
continue to use the notation for the decision tree as introduced in 
Section \ref{id28}. In contrast to the decision tree for the groups of 
order $2^8$, the identification for groups of order $2^9$ does not use 
individual functions $f_v$ for each vertex of the tree.

Let $\lambda(G)$ denote the last non-trivial term of the lower $p$-central
series of a $p$-group $G$. The first steps of the function $f_v$ are:
\begin{enumerate}
\item[$\bullet$] Step (1):
The rank of the group $G$.
\item[$\bullet$] Step (2):
The orders $o(g)$ of the elements $g\in G$.
\item[$\bullet$] Step (3):
For each conjugacy class $g^G$ of $G$, the pair $(|g^G|, o(g))$.
\item[$\bullet$] Step (4):
The abelian invariants of the successive quotients in the derived series
of $G$.
\item[$\bullet$] Step (5):
$\ID(G/\lambda(G))$.
\end{enumerate}

\noindent 
These steps do not give a fine partition of the groups of order $2^9$, but
they do provide a useful splitting of all groups in more manageable 
subsets. We now go into detail with our further invariants.
\bigskip

\noindent
{\bf Conjugacy classes and power-maps.}
We first apply a variation of the conjugacy class clustering of
Section \ref{id28}. Note that we have determined the conjugacy 
classes of $G$ for Step (3) in the decision tree; denote these 
by $C_1, \ldots, C_k$.

\begin{enumerate}
\item[$\bullet$] Step (6):
We determine $C_i^3 = \{ g^3 \mid g \in C_i \}$ and $j$
with $C_i^3 \subseteq C_{j}$. The map $\varphi : \{ 1, \ldots, k\} \ra
\{ 1, \ldots, k\} : i \ms j$ is a permutation, since $\gcd(|G|, 3) = 1$.
The cycle type of $\varphi$ is an isomorphism invariant. 
\item[$\bullet$] Step (7):
We proceed with the conjugacy class clusters and power-maps as described
in Section \ref{id28}.
\end{enumerate}

\begin{rem}
The steps described so far identify $99.9\%$ of all groups of 
order $2^9$; that is, roughly $10\,000$ groups remain.
\end{rem}
\bigskip

\noindent
{\bf Sibling related invariants.}
The following invariants use a version of the \subquot equivalence concept 
of Section \ref{siblings}. The used version is simplified, since this version
is faster to evaluate and still sufficient for our purposes.

\begin{enumerate}
\item[$\bullet$] Step (8):
Determine $\ID(Q)$ for all proper quotients $Q=G/N$ with $N \leq 
Z(G)$ cyclic.
\item[$\bullet$] Step (9):
Determine $\ID(M)$ for all maximal subgroups $M \leq G$.
\end{enumerate}

\begin{rem}
The steps described so far identify all but $139$ pairs and one triple of
groups of order $2^9$.
\end{rem}
\bigskip

\noindent
{\bf Random isomorphism testing.}
The remaining cases of groups can now be identified effectively using
an isomorphism test, since the sets of possible options for each group
are very small at this point. 

\begin{enumerate}
\item[$\bullet$] Step (10):
We use the random isomorphism technique described in Section \ref{id28}. 
\end{enumerate}

\noindent 
Note that the random isomorphism test produces a correct answer 
in all cases in this application. Note also that Step (10) can be replaced by 
the standard presentation algorithm of the ANUPQ package \cite{ANUPQ}, see 
also \cite{OBr94},  if desired.

\begin{rem}
The results of this test identify all remaining groups of order $2^9$.
\end{rem}

\subsection{Runtimes}

We include an outline of the runtimes used by the algorithms
for identification of groups of order dividing $2^8$ and
groups of order $2^9$.
These calculations used GAP $4.14.0$ on an Intel Core
i7-7700 processor.

For each order dividing $2^8$ we selected $100\,000$ groups
at random (uniformly distributed) from the set of all groups
of the respective order from the SmallGroups Library and
identified them via the algorithm presented in Section \ref{id28}.
For each order
we calculated the average runtime per group. The results
are presented in Table \ref{tab3}.

\begin{table}[h!tb]
\begin{center}
\begin{tabular}{|c|c|c|c|c|c|c|}
\hline
Order of group & $2^4$ & $2^5$ & $2^6$ & $2^7$ & $2^8$ \\
\hline
Average time [ms] & 0.44 & 0.61 & 0.90 & 1.48 & 2.26 \\
\hline
\end{tabular}
\caption{Average runtime of algorithms}
\label{tab3}
\end{center}
\end{table}

We now consider the groups of order $2^9$ stepwise.

\begin{enumerate}
\item[$\bullet$]
Groups that are identified up to Step (7): Each of the $10\,483\,997$
groups was identified once with an average runtime of $10.51ms$ per group.
\item[$\bullet$]
Remaining groups that are identified up to Step (9): Each of $9\,035$
groups was identified once with an average runtime of $48.52ms$ per group.
\item[$\bullet$]
Remaining groups that are identified in Step (10): Each of the $281$ 
groups was identified $100$ times
with an average runtime of $57.13ms$ per group.
\end{enumerate}

\section{Questions and open problems}
\label{quest}

Some interesting problems remain unsolved. We discuss some of them.
\bigskip

\begin{qu}
Let $p$ be a prime. 
\begin{enumerate}
\item[\rm (a)]
Are there (strong) sibling $l$-tuples of arbitrary length $l$ among the 
finite $p$-groups?
\item[\rm (b)]
Are there estimates for the number of (strong) sibling $l$-tuples among 
the groups of order $p^n$?
\end{enumerate}
\end{qu}
\bigskip

\begin{qu}
Which invariants are effective in distinguishing the groups in (strong) 
sibling $l$-tuples?
\end{qu}
\bigskip

We note that Lewis and Wilson \cite{LW12} investigate a similar problem
for families of finite $p$-groups for odd $p$.

\appendix

\section{Table of strong sibling pairs of order $2^9$}
\label{app}

We exhibit the $56$ strong sibling pairs of order $2^9$. Each row in the
Tables \ref{tab4} and \ref{tab5} corresponds to one such pair $(G,H) = 
(SG(2^9, x), SG(2^9,y))$. The first entry in the row is $(x,y)$ 
describing $G$ and $H$. The table includes the rank and the $p$-class of 
$G$ and $H$, the result of the isoclinism test, their automorphism group 
order $|Aut(G)| = |Aut(H)|$, the result of the outer automorphism equivalence
test $Out(G) \cong Out(H)$ and the result of the MIP-test 
$\F_2 G \cong \F_2 H$.

\begin{table}[h!tb] 
\begin{center}
\begin{tabular}{|c|c|c|c|c|c|c|}
\hline
Pair & Rank & $p$-class & Isoclinic & Aut-order & Out-equiv & MIP \\
\hline
(1364, 1368) & 2 & 5 & no & 8192 & yes & no \\
(1365, 1369) & 2 & 5 & no & 8192 & yes & no \\
(1366, 1370) & 2 & 5 & no & 8192 & yes & no \\
(1367, 1371) & 2 & 5 & no & 8192 & yes & no \\
(1680, 1681) & 2 & 5 & yes & 16384 & yes & no \\
(1682, 1683) & 2 & 5 & yes & 16384 & yes & no \\
(1746, 1747) & 2 & 6 & yes & 4096 & yes & no \\
(1748, 1749) & 2 & 6 & yes & 4096 & yes & no \\
(1750, 1751) & 2 & 6 & yes & 8192 & yes & no \\
(1752, 1753) & 2 & 6 & yes & 8192 & yes & no \\
(15676, 15677) & 3 & 3 & yes & 65536 & no & no \\
(15681, 15682) & 3 & 3 & yes & 131072 & no & no \\
(15684, 15685) & 3 & 3 & yes & 131072 & no & no \\
(15826, 15827) & 3 & 3 & yes & 65536 & no & no \\
(27832, 27833) & 3 & 3 & yes & 65536 & no & no \\
(27885, 27886) & 3 & 3 & yes & 65536 & no & no \\
(28340, 28341) & 3 & 3 & yes & 32768 & yes & no \\
(28426, 28427) & 3 & 3 & yes & 32768 & yes & no \\
(28907, 28908) & 3 & 3 & yes & 65536 & no & no \\
(28909, 28910) & 3 & 3 & yes & 65536 & no & no \\
(28911, 28912) & 3 & 3 & yes & 65536 & no & no \\
(28913, 28914) & 3 & 3 & yes & 65536 & no & no \\
(29033, 29034) & 3 & 3 & yes & 32768 & yes & no \\
(29035, 29036) & 3 & 3 & yes & 32768 & yes & no \\
(30917, 30923) & 3 & 3 & yes & 131072 & yes & no \\
(30981, 30983) & 3 & 3 & yes & 131072 & yes & no \\
(32666, 32667) & 3 & 4 & yes & 16384 & yes & no \\
(35554, 35555) & 3 & 4 & yes & 32768 & no & no \\
\hline
\end{tabular}
\caption{Strong sibling pairs of order $2^9$, Part I}
\label{tab4}
\end{center}
\end{table}

\begin{table}[h!tb]
\begin{center}
\begin{tabular}{|c|c|c|c|c|c|c|}
\hline
Pair & Rank & $p$-class & Isoclinic & Aut-order & Out-equiv & MIP \\
\hline
(35556, 35557) & 3 & 4 & yes & 32768 & no & no \\
(35582, 35583) & 3 & 4 & yes & 32768 & no & no \\
(35584, 35585) & 3 & 4 & yes & 32768 & no & no \\
(35663, 35664) & 3 & 4 & yes & 4096 & no & no \\
(42355, 42356) & 3 & 4 & yes & 65536 & no & no \\
(42399, 42400) & 3 & 4 & yes & 65536 & no & no \\
(42482, 42483) & 3 & 4 & yes & 16384 & yes & no \\
(42518, 42519) & 3 & 4 & yes & 32768 & yes & no \\
(42520, 42521) & 3 & 4 & yes & 32768 & yes & no \\
(42577, 42578) & 3 & 4 & yes & 32768 & no & no \\
(42579, 42580) & 3 & 4 & yes & 32768 & no & no \\
(43208, 43209) & 3 & 4 & yes & 16384 & yes & no \\
(43261, 43262) & 3 & 4 & yes & 16384 & yes & no \\
(45383, 45384) & 3 & 4 & yes & 16384 & yes & no \\
(53282, 53284) & 3 & 4 & yes & 8192 & yes & no \\
(53283, 53285) & 3 & 4 & yes & 8192 & yes & no \\
(53300, 53302) & 3 & 4 & yes & 32768 & no & no \\
(53301, 53303) & 3 & 4 & yes & 32768 & no & no \\
(53306, 53307) & 3 & 4 & yes & 16384 & yes & no \\
(53382, 53384) & 3 & 4 & yes & 8192 & yes & no \\
(53383, 53385) & 3 & 4 & yes & 8192 & yes & no \\
(53396, 53398) & 3 & 4 & yes & 8192 & yes & no \\
(53397, 53399) & 3 & 4 & yes & 8192 & yes & no \\
(53484, 53485) & 3 & 4 & yes & 4096 & yes & no \\
(55648, 55649) & 3 & 4 & yes & 65536 & no & no \\
(139350, 139355) & 4 & 3 & yes & 4096 & yes & no \\
(253429, 253432) & 4 & 3 & yes & 16384 & yes & no \\
(253447, 253450) & 4 & 3 & yes & 16384 & yes & no \\
\hline
\end{tabular}
\caption{Strong sibling pairs of order $2^9$, Part II}
\label{tab5}
\end{center}
\end{table}

\bibliography{refs}
\bibliographystyle{plain}

\end{document}